\documentclass[11pt]{amsart}
\usepackage{graphicx}
\usepackage{url}
\usepackage{latexsym}
\usepackage{lineno}
\usepackage{amsfonts,amsmath,amssymb}
\newtheorem{theorem}{Theorem}[section]

\newtheorem{question}[theorem]{Question}

\newtheorem{conjecture}[theorem]{Conjecture}

\newcommand{\av}{\operatorname{Av}}
\newcommand{\Av}{\mathcal{AV}}
\date{\today} 

\usepackage{etoolbox} 

\newcommand*\linenomathpatch[1]{%
  \cspreto{#1}{\linenomath}%
  \cspreto{#1*}{\linenomath}%
  \cspreto{end#1}{\endlinenomath}%
  \cspreto{end#1*}{\endlinenomath}%
}

\linenomathpatch{equation}
\linenomathpatch{gather}
\linenomathpatch{multline}
\linenomathpatch{align}
\linenomathpatch{alignat}
\linenomathpatch{flalign}

\begin{document}

\title[Sets of patterns with long monotone subsequences]{Permutations avoiding sets of patterns with long monotone subsequences}

\author{Mikl\'os B\'ona}
\thanks{This work was supported by a grant from the Simons Foundation (421967, MB)}
\address{Department of Mathematics, University of Florida, $358$ Little Hall, PO Box $118105$,
Gainesville, FL, $32611$-$8105$ (USA)}
\email{bona@ufl.edu}
\author{Jay Pantone}
\thanks{This work was supported by a grant from the Simons Foundation (713579, JP)}
\address{Department of Mathematical and Statistical Sciences, Marquette University, PO Box 1881, Milwaukee WI, 53201-1881 (USA)}
\email{jay.pantone@marquette.edu}
\date{}

\begin{abstract} We enumerate permutations that avoid all but one of the $k$ patterns of length $k$ starting with a  monotone increasing subsequence of length $k-1$.
We compare the size of such permutation classes to the size of the class of permutations avoiding the monotone increasing subsequence of length $k-1$. In most cases, 
we determine the exponential growth rate of these permutation classes, while in the remanining cases, we present strong numerical evidence leading to a  conjectured growth rate.
We also present numerical evidence that suggests a conjecture for  the growth rates of these permutation classes at subexponential precision. Some of these conjectures claim
that the relevant permutation classes have non-algebraic, and in one case, even non-D-finite, generating functions.
\end{abstract}

\maketitle

\section{Introduction}
\label{section:introduction}
 We say that a permutation $p$ {\em contains} the pattern (or subsequence) $q=q_1q_2\cdots q_k$ 
if there is a $k$-element set of indices $i_1<i_2< \cdots <i_k$ such that $p_{i_r} < p_{i_s} $ if and only
if $q_r<q_s$.  If $p$ does not contain $q$, then we say that $p$ {\em avoids} $q$. For example, $p=3752416$ contains
$q=2413$, as the first, second, fourth, and seventh entries of $p$ form the subsequence 3726, which is order-isomorphic
to $q=2413$.  A recent survey on permutation 
patterns by Vatter can be found in \cite{vatter}. Let  $\av_n(q)$ be 
the number of permutations of {\em length} $n$ that avoid the pattern $q$, where the {\em length} of a permutation is the number of entries in it. In general, it is very difficult to compute
the numbers  $\av_n(q)$, or to describe their sequence as $n$ goes to infinity.

However, the special case when $q$ is the monotone increasing pattern of length $k$ is much better understood. This is partly because the Robinson--Schensted correspondence 
maps $(12\cdots k)$-avoiding permutations of length $n$ into pairs of standard Young tableaux of the same shape, on $n$ boxes, and having at most $k-1$ columns. The number of such standard Young tableaux, and
the number of their pairs, was computed by Regev \cite{Regev} at great precision. He proved that  for all $k\geq 2$, the asymptotic equality 
 \begin{equation} \label{regeveq}  \av_n(1234\cdots k) \simeq \lambda_k \frac{(k-1)^{2n}}{n^{(k^2-2k)/2}} \end{equation}
 holds, where $\lambda_k$ is a constant given by a multiple integral. 

In particular, it follows from Regev's results that 
\begin{equation} \label{regevlight} L(12\cdots (k-1)) :=\lim_{n\rightarrow \infty} \left(\av_n (12\cdots (k-1)) \right)^{1/n} = (k-2)^2 ,\end{equation}
 a fact that we will also refer to by saying that the {\em exponential growth rate} of the sequence $\av_n(12\cdots (k-1))$ is
$(k-2)^2$. Note that (\ref{regevlight}) is much easier to prove than (\ref{regeveq}). See Theorem 4.10 in \cite{combperm} for an easy proof of the inequality
$\av_n(12\cdots (k-1))\leq (k-2)^{2n}$, which implies
$L(12\cdots (k-1))\leq (k-2)^2$, and see Lemma 5.3 of \cite{bona-lim} or Theorem 1.3 of \cite{merges-staircases} for different proofs of the inequality $L(12\cdots (k-1))\geq (k-2)^2$.

As monotone patterns are so well understood compared to other patterns, it is worth taking study of permutations that do not avoid 
the pattern  $12\cdots (k-1)$, but satisfy pattern avoidance conditions that significantly restrict the ways in which a permutation can contain
 $12\cdots (k-1)$. If $S$ is a set of patterns, and the permutation $p$ avoids all patterns in $S$, then we will say that $p$ \emph{avoids} $S$, and 
we will write $\av_n(S)$ for the number of such permutations of length $n$, $\Av(S)$ for the set of such permutations of all lengths (such a set is called a \emph{permutation class}) and $\Av_n(S)$ for those such permutations of length $n$. 

Let $A_k$ be the set of $k$ patterns of length $k$ that start with an increasing subsequence of length $k-1$. For instance, \[A_5=\{12345,12354,12453,13452,23451\}.\] 
Note that a permutation $p=p_1p_2\cdots p_n$ avoids   $A_k$ if and only if  the subsequence $p_1p_2\cdots p_{n-1}$ avoids 
$12\cdots (k-1)$. 
Therefore, \begin{equation} \label{Ak-eq} \av_n(A_k) = n \av_{n-1}(12\cdots (k-1)) .\end{equation}

If we remove one element of $A_k$, we find more interesting enumeration problems. Let $A_{k,i}=A_k \setminus \{12\cdots (i-1)(i+1)\cdots k i\}$, that is, the set $A_k$ 
with its element ending in $i$ removed. If $p$ avoids $A_{k,i}$, that means that if a pattern of length $k$ that is contained in $p$ starts with an increasing subsequence of length $k-1$,
then the last entry of that pattern has to be its $i$th largest entry. It is clear that for each $i\leq k$, the chain of inequalities $(k-2)^2 \leq L(A_{k,i}) \leq (k-1)^2$ holds, 
since if a permutation avoids the increasing pattern of length $k-1$, then it avoids $A_k$, and for all $i$, the set $A_{k,i}$ contains either the monotone pattern $12\cdots k$, 
or the pattern $12\cdots k(k-1)$, each of which are avoided by fewer than $(k-1)^{2n}$ permutations of length $n$. See Theorem 4.10 and Exercise 4.1 in \cite{combperm} for simple 
proofs of these upper bounds. The interesting question is {\em where} in the interval 
$[(k-2)^2,(k-1)^2]$ are the growth rates $L(A_{i,k})$ located.  

Our goal in this paper is to determine the exponential growth rate $L(A_{k,i})$ of the sequence $\av_n(A_{k,i})$, for each $i\leq k$. These growth rates fall into three categories, depending on what
$i$ is. For $2\leq i\leq k-1$, we prove that $L(A_{k,i})=(k-2)^2$, so avoiding $A_{k,i}$ is just as hard (in the exponential sense) as avoiding $12\cdots (k-1)$. For $i=k$, we prove that
 $L(A_{k,k})=(k-2)^2+1$. The case of $i=1$ is the most difficult. Note that in the case of $k=3$, the set of patterns to avoid is just $A_{3,1}=\{123,132\}$, and it is well
known (\cite{awalk}, Exercise 14.1) that $\av_n(123,132)=2^{n-1}$. So in this case, $L(A_{k,1})=(k-2)^2+1$. On the other hand, if $k=4$, then the set of patterns to avoid is
$A_{4,1}=\{1234, 1243, 1342\}$, and the generating function of permutations avoiding that set of patterns is given in \cite{callan} and could alternatively be computed using the $\mathcal{C}$-machine framework in~\cite{c-machines}. It follows from that generating function
that $L(A_{4,1})=2+\sqrt{5}\approx 4.236$.
We are unable to rigorously compute $L(A_{5,1})$, but we give extremely strong experimental evidence that $L(A_{5,1}) = 9$, corresponding in this $k=5$ case to $(k-2)^2$.

\section{When $2\leq i \leq k-1$}

If a permutation $p$ avoids $A_{k,i}$, but contains an increasing subsequence of length $k-1$, then the set of entries of $p$ that follow the last entry of that increasing subsequence
is very restricted. This leads to the following theorem. 

\begin{theorem} \label{mosti}
For all $k\geq 3$, and all $2\leq i \leq k-1$, the equality 
\[L(A_{k,i}) = (k-2)^2 \]
holds. 
\end{theorem} 

\begin{proof}  
Let $p=p_1p_2\cdots p_n \in \Av_n(A_{k,i})$. 
For any entry $p_h$ of $p$, let the {\em rank} of $p_h$ be the length of the longest increasing subsequence of $p$ that ends in $p_h$. 
 We define two words over the alphabet $\{1,2,\cdots ,k-1\}$. Let $w(p)$ be the word whose $h$th letter is the {\em rank} of $p_h$,
and let $z(p)$ be the word whose $h$th letter is the rank of $h$ as an entry in $p$. 

Note that $i<k$, so it follows that each $p\in \Av_n (A_{k,i})$ avoids the increasing pattern of length $k$, so all entries of $p$ 
have rank less than $k$. (This is not true when $i=k$, and that is why that case will have to be treated separately in Section~\ref{section:i-equals-k}.)
 Therefore, $w(p)$ and $z(p)$ will indeed be words over the mentioned
alphabet. Furthermore, the map $p\rightarrow (w(p),z(p))$ is injective, since entries of a fixed rank $j$ form a decreasing sequence, hence $p$ can be recovered from its image
$(w(p),z(p))$.

 Let $p\in \Av_n (A_{k,i})$, and again write $p=p_1p_2\cdots p_n$. 
Let us take a closer look at $w(p)$. Let $j$ be the smallest index such that $w(p)_j=k-1$. (If there is no such $j$, then $p$ avoids the increasing pattern $12\cdots (k-1)$, and so the number of possibilities for $p$ is less than $(k-2)^{2n}$ as we explained in Section~\ref{section:introduction} following equation (\ref{regevlight}).) That means that there is an increasing  subsequence of length $k-1$ of  $p$ ending in $p_j$.  Let $a_1<a_2<\cdots <a_{k-1}=p_j$ be such a subsequence. If there are several such subsequences, then choose the one such that $a_{k-2}$ is maximal, then $a_{k-3}$ is maximal, and so on. 
 Then for all entries $x$ on the right of $p_j$, the inequalities $a_{i-1}<x<a_i$ must hold, otherwise $a_1a_2\cdots a_{k-1}x$ is a forbidden pattern.
 That means that all such entries of $p$ are of rank $i$ or higher, 
so the last $n-j$ letters of $w(p)$ are $i$ or larger. Therefore, the number of possible words $w(p)$ is at most $\sum_{j=1}^n (k-2)^{j-1} (k-i)^{n-j} \leq n(k-2)^{n-1}$.
Note that in the last estimate, we used the fact that $i>1$.

Now consider $z(p)$. As $p_j$ is the leftmost entry  of $p$ that is of rank $k-1$, and all subsequent entries of $p$ are between $a_{i-1}$ and $a_i$ in value, it follows that 
all entries of $p$ that are of rank $k-1$ except for $p_j$ must be between $a_{i-1}$ and $a_i$ in value. Therefore, if $t\notin (a_{i-1},a_i)$, then the $t$th letter of $z(p)$ cannot
be $k-1$ (except once, the $p_j$th letter), while if $t\in (a_{i-1},a_i)$, then the $t$th letter of $z(p)$ cannot be  $i-1$. Indeed, let us assume the entry $t$ of $p$ is of rank $i-1$, 
and that $a_{i-1}<t<a_i$ holds. Then $t$ must be located on the left of $a_{i-1}$ (since entries of the same rank form a decreasing subsequence),
 and hence, on the left of $a_i$. So there is an increasing subsequence in $p$ that is of length $i$ and
whose last two entries are $t$ and $a_i$, contradicting the maximality requirement of the preceding paragraph. 
Therefore, setting $m=a_{i} - a_{i-1} -1$, we have fewer than $n^3 (k-2)^m (k-2)^{n-m-3} <n^3 (k-2)^{n-3} $ possibilities for $z(p)$. 
Indeed, once we know the locations of $p_j$, $a_{i-1}$ and $a_i$, we know that in those positions, $z(p)$ has entries $k-1$, $i-1$ and $i$, respectively. For each of the remaining
$n-3$ letters of $z(p)$, we have $k-2$ possibilities, because for some of them, $k-1$ is not a possibility, and for the rest of them, $i-1$ is not a possibility. 

This implies that the total number of possibilities for the pair $(w(p),z(p))$ is less than $n^4 (k-2)^{2n}$, which proves our claim as we have already shown in the introduction
that $L(A_{k,i} ) \geq (k-2)^2$.
\end{proof}

\section{When $i=k$}
\label{section:i-equals-k}
The case of $i=k$ leads to a different result. 
\begin{theorem} \label{largesti}  For $k\geq 3$, the equality 
\[L(A_{k,k})=(k-2)^2+1 \]
 holds.
\end{theorem}

In this section, we will assume that the reader is familiar with the Robinson--Schensted correspondence. Readers who wish to learn about that correspondence can consult Chapter 3 of 
\cite{sagan} for a thorough introduction, or Section 7.1 of \cite{combperm} for a survey of some relevant facts. 

\begin{proof} Let $p=p_1p_2\cdots p_n \in \Av_n(A_{k,k})$, and let $P(p)$ and $Q(p)$ be the $P$-tableau and $Q$-tableau of $p$, obtained by the Robinson--Schensted correspondence. 
If $p_{i+1}$ is the leftmost entry of $p$ that is of rank $k-1$, then $p_{i+1}p_{i+2} \cdots p_n$ must be an increasing subsequence. 
This means that the last $n-i$ positions that are filled in both tableaux are the $(k-1)$th, $k$th, $\cdots$, last positions of the first row. 
In particular, this implies that in $Q(p)$, these positions are filled with the entries $i+1,i+2,\cdots ,n$. 

Note that this means that $P(p)$ and $Q(p)$ consist of two parts. One part consists of the first $k-2$ columns, which we will call the {\em front}, and the remaining columns, which are all of height one. We will call this second part the {\em tail}.  As we said above, in $Q(p)$, the content of this second part is known. 
Therefore, to specify $p$, it suffices to select the content of the  tail of $P(p)$ in ${n\choose n-i}$ ways, then select the front of $P(p)$ and $Q(p)$ in 
at most $(k-2)^{2i}$ ways. This leads to the upper bound 

\begin{eqnarray*} 
\av_n(A_{k,k}) \leq \sum_{i=k-2}^{n-1}  {n\choose n-i} (k-2)^{2i} \\
\leq  \sum_{i=0}^{n}  {n\choose n-i} (k-2)^{2i}  \leq \left ( (k-2)^2 +1 \right )^n  .
\end{eqnarray*}

We still have to show that the exponential order of the sequence $\av_n(A_{k,k})$ is at least $(k-2)^2+1$. In order to do so, we construct $A_{k,k}$-avoiding permutations of length $n$
as follows. We choose an integer $\ell $ such that $0\leq \ell \leq n$. Then we select an $\ell$-element subset $S$ of $[n]$. Next, we select a permutation $\pi$ on $[n]-S$ that avoids
$12\cdots (k-1)$, and we postpend $\pi$ with the entries of $S$, written in increasing order, to get the permutation $p$. Note that $p$ avoids $A_{k,k}$. Indeed, patterns in $A_{k,k}$ 
increase until they reach an entry of rank $k-1$, then decrease. This is not possible in $p$, since the only entries of rank $k-1$ or higher are in the last $\ell$ positions, and $p$ is increasing in all those positions. 

For a given $\ell$, the number of ways in which we can carry out the above steps is ${n\choose \ell} \av_{n-\ell}(12\cdots (k-1))$. For a given choice of $\ell$, each permutation $p$ will
be obtained at most once, so we will obtain at least
\[\frac{1}{n+1} \sum_{\ell=0}^n {n\choose \ell} \av_{n-\ell}(12\cdots (k-1)) \] different permutations of length $n$ that avoid $A_{k,k}$. The division by $n+1$ is necessary
because different choices of $\ell$ can lead to the same $p$. 

Finally note that it follows from (\ref{regeveq}), substituting $k-1$ in the place of $k$, that there exists a constant $K_{k-1}$ such that for all positive integers $n$, the inequality
\[\av_n(12\cdots (k-1))\geq K_{k-1} \frac{(k-2)^{2n}}{n^{(k^2-4k+3)/2}} \] holds. Comparing the last two displayed expressions, we see that we have constructed at least
\[\frac{K_{k-1}}{n^{(k^2-4k+3)/2}(n+1)} \sum_{\ell =0}^n {n\choose \ell} (k-2)^{2(n-\ell)} = \frac{K_{k-1}}{n^{(k^2-4k+3)/2}(n+1)} ((k-2)^2+1 )^n\] elements of
$\Av_n(A_{k,k})$, proving that the exponential order of $\av_n(A_{k,k})$ is at least $(k-2)^2+1$. 
\end{proof}

\section{An injection}
\label{section:injection}
It follows from Theorems \ref{mosti} and \ref{largesti} that if $n$ is large enough, then $\av_n(A_{k,k-1}) < \av_n(A_{k,k})$. 
In this section, we prove that  $\av_n(A_{k,k-1}) \leq  \av_n(A_{k,k})$ for {\em every} $n$. The interest of this result lies in its proof, which is
 by a very simple injective map. It is rare that nontrivial inequalities between permutation class sizes can be proved injectively. 

\begin{theorem} For all positive integers $n$, and all $k\geq 3$, the inequality 
\[\av_n(A_{k,k-1}) \leq  \av_n(A_{k,k}) \]
holds. 
\end{theorem}

\begin{proof}
Let $p\in \Av _n (A_{k,k-1})$.  Let $p_i$ be the leftmost entry of $p$ that is of rank $k-1$ if such an entry $p_i$ exists. Then the entries $p_{i+1},p_{i+2},\cdots ,p_n$ must all be of rank $k-1$, and therefore, 
the subsequence $p_ip_{i+1}\cdots p_n$ is decreasing. Now we define a map $f: \Av _n (A_{k,k-1}) \rightarrow  \Av _n (A_{k,k})$ by setting $f(p)=p$ if $p$ does not have an entry of 
rank $k-1$, and $f(p)=p_1p_2\cdots p_{i-1}p_np_{n-1} \cdots p_i$ otherwise. In other words, $f(p)$ is obtained by reversing the decreasing subsequence $p_ip_{i+1}\cdots p_n$ of
$p$, that consists of entries of rank $k-1$ in $p$. 

It is then clear that $f(p)$ avoids $A_{k,k}$, since the only entries of rank $k-1$ or higher in $f(p)$ are those in the last $n - i + 1$ positions, and those entries are all in increasing order. 
Furthermore, $f$ is injective, since given $r\in \Av_n(A_{k,k})$, we can look at the maximal (non-extendible)  increasing subsequence at the end of $r$. That subsequence contains exactly one entry 
$x$ of rank $k-1$. The only preimage of $r$ under $f$ can be obtained by reversing the subsequence of $r$ that starts in $x$ and goes all the way to the end of $r$. 
(Note that $f$ is not a bijection, because reversing that subsequence of $r$ will not always result in a permutation in $ \Av _n (A_{k,k-1})$.)
\end{proof}

\section{When $i=1$}
\label{section:i-equals-1}

While we are not able to rigorously compute $L(A_{5,1})$, in this section we describe how to instead rigorously compute the first $642$ terms of the counting sequence of $A_{5,1}$-avoiding permutations, from which we derive very strong numerical evidence that $L(A_{5,1}) = 9$.

\begin{conjecture} The equality $L(A_{5,1}) = 9$ holds. \end{conjecture}

The Combinatorial Exploration paradigm developed by Albert, Bean, Claesson, Nadeau, Pantone, and Ulfarsson ~\cite{comb-exp} is a computational framework for enumerating combinatorial objects\footnote{All of the relevant code is open-source and can be found on GitHub~\cite{comb-exp-code}.}. Combinatorial Exploration is experimental in the sense that you do not know ahead of time whether it will succeed. However, when it does succeed, the output is a fully rigorous structural description of the class in the form of a \emph{combinatorial specification}. From this combinatorial specification, the framework automatically produces a polynomial-time counting algorithm for the class, a system of equations that the generating function for the class must satisfy, as well as other products that are not relevant here.

When applied to the permutation class $A_{5,1}$, Combinatorial Exploration finds a combinatorial specification in a few minutes, and a more favorable combinatorial specification in a few hours.\footnote{It is often the case that one can spend additional computing time to find a combinatorial specification whose polynomial-time counting algorithm has a lower polynomial degree, and is thus considerably faster.} The system of equations involves a main variable $x$ and two additional ``catalytic'' variables $y$ and $z$, and we do not know of any methods to solve it exactly, nor to extract from it any information about the asymptotic behavior of its solution. We used the resulting polynomial-time algorithm to compute the number of permutations of length $n$ in $A_{5,1}$ for $n \leq 641$ in about 20 hours. These initial terms of the counting sequence can be experimentally analyzed in several ways.

Firstly, we can use them to attempt to make a conjecture about the generating function of the counting sequence $\av_n(A_{5,1})$. There are many software packages that do this kind of computation (e.g., \texttt{Gfun}~\cite{gfun} in Maple), all with various strengths and weaknesses. We used a package called \texttt{GuessFunc}~\cite{guessfunc-bibtex} written in Maple by the second author. This package tries to fit the given initial terms to a rational, algebraic, D-finite, or differentially algebraic generating function; briefly, a generating function $f(x)$ is D-finite if it satisfies a non-trivial linear differential equation
\[
	p_k(x)f^{(k)}(x) + p_{k-1}(x)f^{(k-1)}(x) + \cdots + p_0(x)f(x) + q(x) = 0
\]
where the coefficients $p_i(x)$ and $q(x)$ are polynomials, and $f(x)$ is differentially algebraic if there is a polynomial $P$ such that
\[
	P(x, f(x), f'(x), \ldots, f^{(k)}(x)) = 0.
\]
\texttt{GuessFunc} works, roughly, by assuming that the generating function has a particular form (e.g., D-finite of differential order $3$ with polynomial coefficients of degree $12$), and using the known initial terms of the counting sequence to set up a corresponding linear system of equations. If that system has a solution, that solution leads to a conjectured generating function.

Using the 642 initial terms, we were unable to make a conjecture that the generating function of $A_{5,1}$ has any of these forms. While this is not dispositive, it implies that if $A_{5,1}$ were D-finite, for example, either the differential order or the maximum degree of one of the polynomial coefficients would need to be quite large.

\begin{conjecture}
	The generating function for $\Av(A_{5,1})$ is not D-finite.
\end{conjecture}

Secondly, and more relevant to the pursuits of this work, we can apply the method of differential approximation~\cite{diff-approx, DiffApprox-bibtex} to empirically estimate the asymptotic growth of the counting sequence. The method of differential approximation constructs a collection of D-finite generating functions whose initial power series coefficients match the known initial terms of the given counting sequence (later terms are not expected to match). Then, the asymptotic behaviors of the D-finite generating functions are studied in aggregate in order to make predictions about the asymptotic behavior of the counting sequence in question. When tested on sequences whose asymptotic growth is independently already known, differential approximation shows a remarkable ability to provide very precise estimates.

Using the first $200$ terms of the counting sequence of $A_{5,1}$, differential approximation predicts that the dominant singularity of its generating function (that is, the one closest to the origin) is located at
\[
	x_c \approx 0.1111111112
\]
indicating an exponential growth rate of $1/x_c \approx 9$ with very high precision. Further, it approximates the value of the corresponding critical exponent (a property of a given singularity) to be
\[
	\alpha \approx 1.9999999990
\]
indicating a polynomial term of $n^{-1-\alpha} \approx n^{-3}$. As a result, we have strong  experimental evidence for the following. 

\begin{conjecture}
	There exists a constant $C$ such that 	
	\[
		\av_n(A_{5,1}) \sim  C\cdot9^nn^{-3}.
	\]
	The value of $C$ appears to be roughly $0.47$.
\end{conjecture}
Such asymptotic growth, if verified, would not rule out the possibility that the generating function could be D-finite.

\section{Further Experimental Results}
\label{section:further-experimental}

In Section~\ref{section:i-equals-1}, we presented experimental evidence that the asymptotic growth of $A_{5,1}$ has the form $C \cdot 9^nn^{-3}$ and that the generating function of $A_{5,1}$ is non-D-finite. Combinatorial Exploration successfully produces combinatorial specifications for the other four classes of interest for $k=5$, allowing us to compute many initial terms of the counting sequences. In this section we quickly summarize the experimental results we find for $A_{5,2}$, $A_{5,3}$, $A_{5,4}$, and $A_{5,5}$, each of which now has a known exponential growth rate due to the previous sections, as well as for $A_{6,1}$.

\subsection{$A_{5,2}$}

Combinatorial Exploration produces a combinatorial specification for the permutation class $A_{5,2}$ in about $5$ hours. The resulting polynomial-time enumeration algorithm is slower than the one we found for $A_{5,1}$. We are only able to compute $91$ terms in the counting sequence in about $5$ hours using $300$gb of memory. \texttt{GuessFunc} provides no conjecture for the generating function of this sequence.  However, based on differential approximation, we make the following conjecture.

\begin{conjecture} 
There exists a constant $C$ such that \[
\av_n (A_{5,2}) \sim C \cdot 9^n n^{-3}.
\]
\end{conjecture}
Differential approximation also shows a subdominant singularity (i.e., a singularity that is not a singularity closest to the origin) in the area of $x \approx 0.18750 = 3/16$. In future subsections, we will only mention subdominant singularities in cases where they are detected.

\subsection{$A_{5,3}$}

For $A_{5,3}$, we are able to calculate the first $130$ terms of the counting sequence. \texttt{GuessFunc} is unable to produce a conjecture for the generating function of $A_{5,3}$. Based again on differential approximation, we make the following conjecture.

\begin{conjecture} 
There exists a constant $C$ such that \[
\av_n(A_{5,3}) \sim	C \cdot 9^n n^{-3}.
\]
\end{conjecture}

This time, differential approximation  suggests a subdominant singularity in the area of $x \approx 0.2$.

\subsection{$A_{5,4}$}

In this case, we can compute the first $444$ terms of the counting sequence in 13 hours, using 182gb of memory. Unlike the previous cases, \texttt{GuessFunc} predicts using the first $160$ terms that the generating function of $A_{5,4}$ is D-finite with differential order $6$ and maximum polynomial degree $17$. (We should note here that a D-finite generating function often satisfies many different linear differential equations, and that there is normally a tradeoff in which lowering the differential order results in a higher polynomial degree, and vice versa.) We will not reproduce the differential equation here due to its size. We should note that the output of \texttt{GuessFunc} is merely a conjectured generating function, although we have a high degree of confidence in it because it was found using only the first $160$ terms, and then matched nearly $300$ additional terms. Conjectured generating functions can sometimes be rigorously confirmed using a ``guess-and-check'' approach if other information is already rigorously known, but that is not the case here.

Using differential approximation once again, we predict an  asymptotic growth of the following form.

\begin{conjecture}
There exists a constant $C$ such that 
\[
	\av(A_{5,4}) \sim C \cdot 9^nn^{-3}.
\]
\end{conjecture}

\subsection{$A_{5,5}$}

For $A_{5,5}$ we have found the first $425$ terms of the counting sequence using about 6.5 hours and 107gb of memory. Like the previous case, \texttt{GuessFunc} conjectures that the generating function is D-finite, this time with differential order $3$ and maximum polynomial degree $8$ and using only the first $55$ terms. This one is small enough to print: the generating function $F(x)$ appears to satisfy the equation
\begin{align*}
	&x^3(x-1)(5x-2)(10x-1)(2x-1)^2F'''(x)\\
	&\quad + x^2(2x-1)(650x^4 - 1375x^3 + 909x^2 - 227x+16)F''(x)\\
	&\quad + x(2x-1)(800x^4 - 1850x^3 + 1277x^2 - 339x +28)F'(x)\\
	&\quad + (200x^5 - 700x^4 + 716x^3 - 329x^2 + 76x -8)F(x)\\
	&\quad + 2(5x-2)^2 = 0
\end{align*}

Differential approximation predicts an asymptotic growth of the following form.

\begin{conjecture} There exists a constant $C$ such that 
\[
	 \av_n(A_{5,5}) \sim  C \cdot 10^nn^{-4}.
\]
\end{conjecture}

The presence of the $10x-1$ factor in the coefficient of $F'''(x)$ in the conjectured differential equation indicates the possibility (but not the certainty) of an exponential growth rate of $10$ for the counting sequence, although we know already from Theorem~\ref{largesti} that the growth rate is indeed $10$.

\subsection{$A_{6,1}$} 

Lastly, for $A_{6,1}$ we were only able to compute $71$ terms of the counting sequence before running out of memory. We were unable to conjecture a generating function, but differential approximation suggests that the growth rate is $16$. Although our confidence is not high, we announce the following conjecture. 

\begin{conjecture}
There exists a constant $C$ such that 
	\[ \av_n (A_{6,1}) \sim  C \cdot 16^n n^{-13/2}.
\]
\end{conjecture}

\section{Further directions}
The strong computational evidence obtained in this paper about the subexponential factor of the asymptotic growth of our sequences raises several intriguing questions. 
Answering them could shed some light on analogous problems for longer patterns as well. 

First, we saw in Sections~\ref{section:i-equals-1} and~\ref{section:further-experimental} that if $1\leq i \leq 4$, then there is strong numerical evidence to suggest
that $\av_n(A_{5,i}) \sim C_i \cdot  9^n n^{-3}$, where $C_i$ is some positive constant. This would mean that
$\av_n(A_{5,i})$ is just a linear factor larger than $\av(12\cdots (k-1))$, and therefore, by formula (\ref{Ak-eq}),  it only differs from $\av_n(A_5)$
in a constant factor. This result would be surprising on its own, and it would also imply that the generating function 
of the sequence $\av_n(A_{5,i})$ is not algebraic if $1\leq i\leq 4$.  Note that the behavior of $\av_n(A_{6,1})$ seems to be very similar. If its growth rate is indeed $C\cdot 16^n n^{-13/2}$ as we suggested at the end of Section 6, then the growth rate of that sequence is one linear factor larger than that of  the sequence $\av_n(12345)$. 

Second, in Section 6, we saw data suggesting that $\av_n(A_{5,5}) \sim C_5 \cdot 10^n n^{-4}$. We know from 
\cite{Regev} that $\av_n(1234)\sim C\cdot 9^n n^{-4}$.  Perhaps our proof of Theorem \ref{largesti} could be refined to prove this more precise asymptotic formula for $\av_n(A_{5,5})$. Such a result would imply that the generating function of the sequence is not algebraic. 

The third, and perhaps most interesting, question is the exponential growth rate of the sequence $\av_n(A_{k,1})$ for general $k$. We saw in the introduction that $L(A_{3,1})=(3-2)^2+1=2$, and $L(A_{4,1}) = 2+\sqrt{5} $, 
which is between 4 and 5, that is, the values of $(k-2)^2$ and $(k-2)^2+1$. Numerical evidence obtained in this 
paper suggests that $L(A_{5,1})=9=(5-2)^2$ and $L(A_{6,1})=16=(6-2)^2$. This raises the following intriguing question.

\begin{question} Is it true that if $k\geq 5$, 
then $L(A_{k,1})=(k-2)^2$?
\end{question}


\begin{thebibliography}{99}

\bibitem{comb-exp}
M.~H. Albert, C.~Bean, A.~Claesson, {\'E}.~Nadeau, J.~Pantone, and H.~Ulfarsson.
\newblock Combinatorial {E}xploration: An algorithmic framework for
  enumeration.
\newblock \url{https://arxiv.org/abs/2202.07715}, 2022.


\bibitem{c-machines}
\newblock M.~H.~Albert, C.~Homberger, J.~Pantone, N.~Shar, and V.~Vatter, Generating permutations with restricted containers.
\newblock {\em J. Combin. Theory Ser. A} {\bf 157}
(2018), 205–232.

\bibitem{merges-staircases}
\newblock M.~H.~Albert, J.~Pantone, and V.~Vatter, On the growth of merges and staircases of permutation classes.
\newblock {\em Rocky Mountain J. Math.} {\bf 49}
(2019), no. 2, 355-367.

\bibitem{comb-exp-code}
C. Bean, J. S. Eliasson, T. K. Magnusson, \'E. Nadeau, J. Pantone, H. Ulfarsson,
\newblock{Tilings: Combinatorial Exploration for permutation patterns}.
\newblock \url{https://github.com/PermutaTriangle/Tilings}, June 2021.
\newblock DOI: \url{https://doi.org/10.5281/zenodo.5810636}.


\bibitem{bona-lim}
\newblock  M. B\'ona, The limit of a Stanley-Wilf sequence is not always rational, and layered patterns beat monotone patterns.
\newblock {\em J. Combin. Theory Ser. A} {\bf 110}
 (2005), no. 2, 223–235.

\bibitem{combperm}  M. B\'ona, Combinatorics of Permutations, 2nd edition, CRC Press, 2012.

\bibitem{awalk} M. B\'ona, A Walk Through Combinatorics, 4th edition, World Scientific, 2016.

\bibitem{callan} D. Callan, T. Mansour,  Five subsets of permutations enumerated as weak sorting permutations,  {\em Southeast Asian Bull. Math.} {\bf 42} (2018), no. 3, 327--340.

\bibitem{diff-approx}
A.~J. Guttmann,
\newblock Asymptotic analysis of power-series expansions.
\newblock In C. Domb and J. L. Lebowitz, editors, {\em Phase Transitions and Critical Phenomena, Vol. 13}, pages 1--234.
\newblock Academic Press, London, England, 1989.

\bibitem{DiffApprox-bibtex}
J. Pantone,
\newblock {DiffApprox: A Maple library to predict the asymptotic behavior of counting sequences given some initial terms}.
\newblock \url{https://github.com/jaypantone/DiffApprox}, December 2021.
\newblock DOI: \url{https://doi.org/10.5281/zenodo.5810652}.

\bibitem{guessfunc-bibtex}
J. Pantone,
\newblock GuessFunc: A Maple library to guess the generating function of a counting sequence given some initial terms.
\newblock \url{https://github.com/jaypantone/guessfunc}, December 2021.
\newblock DOI: \url{https://doi.org/10.5281/zenodo.5810636}.


\bibitem{Regev}
A. Regev, Asymptotic values for degrees associated with
strips of Young diagrams, {\em Advances in Mathematics},  {\bf 41} (1981), 115--136.

\bibitem{sagan} B. Sagan, The Symmetric Group: Representations, Combinatorial Algorithms, and Symmetric Functions,
(Graduate Texts in Mathematics, Vol. 203) 2nd Edition, Springer, 2001. 

\bibitem{gfun}
B.~Salvy and P.~Zimmermann,
\newblock {GFUN}: A {M}aple package for the manipulation of generating and
  holonomic functions in one variable.
\newblock {\em ACM Trans. Math. Softw.}, 20(2):163--177, June 1994.

\bibitem{vatter} V. Vatter, Permutation classes. In: Handbook of Enumerative Combinatorics, Mikl\'os B\'ona, editor,
CRC Press, 2015.
\end{thebibliography}
\end{document}